\documentclass[preprint,12pt]{elsarticle}

\usepackage{amssymb}

\newtheorem{thm}{Theorem}

\newdefinition{dfn}{Definition}
\newdefinition{ex}{Example}
\newproof{proof}{Proof}

\begin{document}

\begin{frontmatter}



\title{BVP's with discontinuities at finite number interior points and spectral parameter in the boundary conditions  }


\author[rvt]{K. Aydemir}
\ead{kadriye.aydemir@gop.edu.tr} \cortext[cor1]{Corresponding Author
(Tel: +90 356 252 16 16, Fax: +90 356 252 15 85)}

\address[rvt]{Department of Mathematics, Faculty of Arts and Science, Gaziosmanpa\c{s}a University,\\
 60250 Tokat, Turkey}

\begin{abstract}
In this study we are concerned with a  class of generalized  BVP' s
consisting of eigendependent boundary conditions and supplementary
transmission conditions at finite number interior points. By
modifying some techniques of classical Sturm-Liouville theory
 and  suggesting own approaches we find asymptotic formulas for the eigenvalues and eigenfunction.
\end{abstract}


\end{frontmatter}


\section{Introduction}

In this study we shall investigate a new class  of BVP's which
consist of the Sturm-Liouville equation
\begin{equation}\label{1}
\mathcal{L}(u):=-\rho(x)u^{\prime \prime }(x)+ q(x)u(x)=\lambda
u(x), x\in \Omega
\end{equation}
together with eigenparameter-dependent boundary conditions at end
 points $x=a, b$
\begin{equation}\label{2}
\mathcal{L}_1(u):=\alpha_1u(a)-\alpha_2u'(a)-\lambda(\alpha_3u(a)-\alpha_4u'(a))=0,
\end{equation}
\begin{equation}\label{3}
\mathcal{L}_2(u):=\beta
_1u(b)-\beta_2u'(b)+\lambda(\beta_3u(b)-\beta_4u'(b))=0,
\end{equation}
and  the transmission conditions at finite interior points $\xi_i
\in (a,b), i=1,2,...r$
\begin{equation}\label{4}
\pounds_i(u)=\alpha^{-}_{i1}u'(\xi_i-)+\alpha^{-}_{i0}u(\xi_i-)+\alpha^{+}_{i1}u'(\xi_i+)+\alpha^{+}_{i0}u(\xi_i+)=0,
\end{equation}
\begin{equation}\label{5}
\mathfrak{L}_i(u)=\beta^{-}_{i1}u'(\xi_i-)+\beta
a^{-}_{i0}u(\xi_i-)+\beta^{+}_{i1}u'(\xi_i+)+\beta^{+}_{i0}u(\xi_i+)=0,
\end{equation}
where  $\Omega=\bigcup\limits_{i=1}^{r+1}(\xi_{i-1}, \xi_{i}), \
a:=\xi_0, \ b:=\xi_{r+1}, \ \rho(x)=\rho_i^{2}>0  \ \textrm{for} \ x
\in \Omega_i:= (\xi_{i-1}, \xi_{i}), \ i=1,2,...r+1 $, the potential
$q(x)$ is real-valued function which continuous in each of the
intervals $(\xi_{i-1}, \xi_{i})$, and has a finite limits $q(
\xi_{i}\mp0)$, $\lambda$ \ is a complex spectral parameter, \
$\alpha_{k}, \ \beta_{k}\ (k=1,2,3,4), \ \alpha^{\pm}_{ij}, \
\beta^{\pm}_{ij} \ (i=1,2,...r\ \textrm{and} \ j=0,1)$ are real
numbers. We want emphasize that the boundary value problem studied
here differs from the standard boundary value problems in that it
contains transmission conditions and the eigenvalue-parameter
appears not only in the differential equation, but also in the
boundary conditions. Moreover the coefficient functions may have
discontinuity at finite interior point. Naturally, eigenfunctions of
this problem may have discontinuity at the finite inner point of the
considered interval. The problems with transmission conditions has
become an important area of research in recent years because of the
needs of modern technology, engineering and physics. Many of the
mathematical problems encountered in the study of
boundary-value-transmission problem cannot be treated with the usual
techniques within the standard framework of  boundary value problem
(see \cite{ji}). Note that some special cases of this problem arise
after an application of the method of separation of variables to a
varied assortment of physical problems. For example, some boundary
value problems with transmission conditions arise in heat and mass
transfer problems \cite{lik}, in vibrating string problems when the
string loaded additionally with point masses \cite{tik}, in
diffraction problems \cite{voito}. Also some problems with
transmission conditions which arise in mechanics (thermal conduction
problems for a thin laminated plate) were studied in  \cite{tite}.
\section{The fundamental solutions and characteristic function}
 With a view to constructing the
characteristic function $\omega(\lambda )$  we shall define  two
fundamental solution
\begin{displaymath} \phi (x,\lambda
)=\left \{
\begin{array}{c}
\phi _{1}(x,\lambda ),  x\in \lbrack a,\xi_1) \\
\phi _{2}(x,\lambda ),  x\in (\xi_1,\xi_2)\\
\phi _{3}(x,\lambda ),  x\in (\xi_2,\xi_3) \\...\\
\phi _{r+1}(x,\lambda ),  x\in (\xi_r,b]\\
\end{array}\right.
\textrm{and} \ \chi(x,\lambda)=\left \{\begin{array}{ll}
\chi _{1}(x,\lambda ),  x\in \lbrack a,\xi_1) \\
\chi _{2}(x,\lambda ),  x\in (\xi_1,\xi_2)\\
\chi _{3}(x,\lambda ),  x\in (\xi_2,\xi_3) \\...\\
\chi _{r+1}(x,\lambda ),  x\in (\xi_r,b]\\
\end{array}\right.
\end{displaymath}
 by special procedure. Let $\phi
_{1}(x,\lambda ) \ \textrm{and} \  \chi _{r+1}(x,\lambda )$ be
solutions of the equation $(\ref{1})$ on $(a,\xi_1) \ \textrm{and }
\ (\xi_n,b)$ satisfying initial conditions
\begin{eqnarray}&&\label{7}
u(a,\lambda)=\alpha_2-\lambda\alpha_4, \ u'(a,\lambda )=\alpha_1 -\lambda a_3 \\
&&\label{10} u(b,\lambda)=\beta_2+\lambda \beta_4, \ u'(b,\lambda
)=\beta_1 +\lambda \beta_3
\end{eqnarray}
respectively. In terms of these solution we shall define the other
solutions $\phi _{i}(x,\lambda ), \ \textrm{and}  \ \chi
_{i}(x,\lambda )$ by initial conditions
\begin{eqnarray}\label{8}
&&\phi _{i+1}(\xi_{i}+,\lambda)
=\frac{1}{\Delta_{i12}}(\Delta_{i23}\phi
_{i}(\xi_{i}-,\lambda)+\Delta_{i24}\frac{\partial\phi
_{i}(\xi_{i}-,\lambda)}{\partial x})
\\ &&\label{9} \frac{\partial\phi _{i+1}(\xi_{i}+,\lambda)}{\partial x}
=\frac{-1}{\Delta_{i12}}(\Delta_{i13}\phi
_{i}(\xi_{i}-,\lambda)+\Delta_{i14}\frac{\partial\phi
_{i}(\xi_{i}-,\lambda)}{\partial x})
\\ \nonumber
\textrm{and}
 \\ &&\label{11}
\chi _{i}(\xi_{i}-,\lambda)
=\frac{-1}{\Delta_{i34}}(\Delta_{i14}\chi
_{i+1}(\xi_{i}+,\lambda)+\Delta_{i24}\frac{\partial\chi
_{i+1}(\xi_{i}+,\lambda)}{\partial x})
\\ && \label{12} \frac{\partial\chi _{i}(\xi_{i}-,\lambda)}{\partial x})
=\frac{1}{\Delta_{i34}}(\Delta_{i13}\chi
_{i+1}(\xi_{i}+,\lambda)+\Delta_{i23}\frac{\partial\chi
_{i+1}(\xi_{i}+,\lambda)}{\partial x})
\end{eqnarray}
respectively, where $\Delta_{ijk} \ (1\leq j< k \leq 4)$ denotes the
determinant of the j-th
 and k-th columns of the matrix
 $$ \left[\begin{array}{cccc}
   \alpha^{+}_{i1} & \alpha^{+}_{i0}  & \alpha^{-}_{i1} & \alpha^{-}_{i0}\\
  \beta^{+}_{i1} & \beta^{+}_{i0} & \beta^{-}_{i1} & \beta^{-}_{i0}
\end{array} %
 \right]  $$
 where $i=1,2,...r.$ Everywhere in below we shall assume that
$\Delta_{ijk}>0$ for all i,j,k. The existence and uniqueness of
these solutions are follows from well-known theorem of ordinary
differential equation theory. Moreover by applying the method of
\cite{ka} we can prove that each of these solutions are entire
functions of parameter $\lambda \in \mathbb{C}$ for each fixed $x$.
Taking into account (\ref{8})-(\ref{12}) and the fact that the
Wronskians $ \omega_i(\lambda):=W[\phi_i(x,\lambda
),\chi_i(x,\lambda )]$ (i=1,2,...r+1) are independent of variable
$x$  we have
\begin{eqnarray*}
\omega_{i+1}(\lambda )&=&\phi_i(\xi_i+,\lambda
)\frac{\partial\chi_i(\xi_i+,\lambda )}{\partial x}-\frac{\partial\phi_i(\xi_i+,\lambda )}{\partial x}\chi_i(\xi_i+,\lambda ) \\
 &=&\frac{\Delta_{i34}}{\Delta_{i12}}(\phi_i(\xi_i-,\lambda )\frac{\partial\chi_i(\xi_i-,\lambda )}{\partial x}
 -\frac{\partial\phi_i(\xi_i-,\lambda )}{\partial x}\chi_i(\xi_i-,\lambda )) \\
&=&\frac{\Delta_{i34}}{\Delta_{i12}} \omega_i(\lambda
)=\prod_{j=1}^{i}\frac{\Delta_{j34}}{ \Delta_{j12}}\omega_1(\lambda
) \ (i=1,2,...r).
\end{eqnarray*}
It is convenient to define  the characteristic function  \
$\omega(\lambda)$ for our problem $(\ref{1})-(\ref{3})$ as
$$
\omega(\lambda):=\omega_1(\lambda)
=\prod_{j=1}^{i}\frac{\Delta_{j12}}{
\Delta_{i34}}\omega_{i+1}(\lambda ) \ (i=1,2,...r).
$$
Obviously, $\omega(\lambda)$ is an entire function. By applying the
technique of \cite{osh6} we can prove that there are infinitely many
eigenvalues $\lambda_{n}, \ n=1,2,...$ of the problem
$(\ref{1})-(\ref{5})$ which are coincide with the zeros of
characteristic function  \ $\omega(\lambda)$.
\section{  Some  results
for eigenvalues }

\begin{thm}
The eigenvalues of the problem $(\ref{1})$-$(\ref{5})$ are consist
of the zeros of the function $w(\lambda ).$\label{t2}
\end{thm}
\begin{thm}
The eigenvalues of the problem $(\ref{1})$-$(\ref{5})$ are are
analytically simple.
\end{thm}
\begin{thm}
The eigenvalues of the problem $(\ref{1})$-$(\ref{5})$ are also
geometrically simple.
\end{thm}
\begin{thm}
The eigenvalues of the problem $(\ref{1})$-$(\ref{5})$  are bounded
below, and they are countably infinite and can cluster only at
$\infty$.
\end{thm}
Now by modifying the standard method we  prove that all eigenvalues
of the problem $(\ref{1})-(\ref{5})$ are real.
\begin{thm}
The eigenvalues of the boundary value transmission problem
$(\ref{1})-(\ref{4})$ are real.
\end{thm}
\section{   Asymptotic formulas
for ,,basic'' solutions}

 Below, for
shorting  we shall use also  notations; $\phi_i(x,\lambda ):=\phi
_{i\lambda}(x), \  \chi_i(x,\lambda ):=\chi_{i\lambda}(x) \
(i=1,2,...r+1).$ We can prove that the next integral and
integro-differential equations are hold for $k=0$ and $k=1.$%
\begin{eqnarray}\label{(4.2)}
\frac{d^{k}}{dx^{k}}\phi_{1\lambda}(x )
&=&(\alpha_2-s^2\alpha_4)\frac{d^{k}}{dx^{k}}\cos
\left[\frac{s\left( x-a\right)}{\rho_1}\right]+
\frac{\rho_1(\alpha_1-s^2\alpha_3)}{s}\frac{d^{k}}{dx^{k}}\sin
\left[\frac{s\left(
x-a\right)}{\rho_1}\right] \nonumber \\
&+&\frac{1}{\rho_1s}\int\limits_{a}^{x}\frac{d^{k}}{dx^{k}}\sin
\left[\frac{s\left( x-y\right)}{\rho_1}\right] q(y)\phi_{1\lambda}(y
)dy, \\ \label{(4.a)} \frac{d^{k}}{dx^{k}}\phi_{(i+1)\lambda}(x )
&=&\frac{1}{\Delta_{i12}}(\Delta_{i23}\phi_{i\lambda}(\xi_i
)+\Delta_{i24}\phi'_{i\lambda}(\xi_i )) \frac{d^{k}}{dx^{k}}\cos
\left[
\frac{s(x-\xi_i)}{\rho_{i+1}}\right]  \nonumber  \\
&-&\frac{\rho_{i+1}}{s
\Delta_{i12}}(\Delta_{i13}\phi_{i\lambda}(\xi_i
)+\Delta_{i14}\phi'_{i\lambda}(\xi_i ))\frac{d^{k}}{dx^{k}}\sin
\left[ \frac{s(x-\xi_i)}{\rho_{i+1}}
\right]  \nonumber \\
&+&\frac{1}{\rho_{i+1}s}\int\limits_{\xi_i}^{x}\frac{d^{k}}{dx^{k}}\sin
\left[ \frac{s\left( x-y\right)}{\rho_{i+1}} \right]
q(y)\phi_{(i+1)\lambda}(y )dy
\end{eqnarray}
for \ \ x \ $\in (a,\xi_1)$ \ \textrm{and} $x \ \in
(\xi_i,\xi_{i+1}), \ i=1,2,...r$ respectively.
\begin{eqnarray}
\frac{d^{k}}{dx^{k}}\chi_{(r+1)\lambda}(x)&=&(\beta_2+s^2
\beta_4)\frac{d^{k}}{dx^{k}} \cos \left[\frac{s\left(
x-b\right)}{\rho_{r+1}} \right] +
\frac{\rho_{r+1}}{s}(\beta_2+s^2\beta_3)\frac{d^{k}}{dx^{k}}\sin
\left[\frac{s\left( x-b\right)}{\rho_{r+1}} \right]
\nonumber\\&+&\frac{1}{\rho_{r+1}s}\int\limits_{x}^{b}\frac{d^{k}}{dx^{k}}\sin
\left[\frac{s\left( x-y\right)}{\rho_{r+1}} \right]
q(y)\chi_{(r+1)\lambda}(y)dy \label{(4.21)}\\
\frac{d^{k}}{dx^{k}}\chi_{i\lambda}(x )
&=&-\frac{1}{\Delta_{i34}}(\Delta_{i14}\chi_{(i+1)\lambda}(\xi_i
)+\Delta_{i24}\chi'_{(i+1)\lambda}(\xi_i ) )\frac{d^{k}}{dx^{k}}\cos
\left[
\frac{s(x-\xi_i)}{\rho_{i}}\right]  \nonumber  \\
&-&\frac{\rho_{i}}{s
\Delta_{i34}}(\Delta_{i13}\chi_{(i+1)\lambda}(\xi_i
)+\Delta_{i23}\chi'_{(i+1)\lambda}(\xi_i ))\frac{d^{k}}{dx^{k}}\sin
\left[ \frac{s(x-\xi_i)}{\rho_{i}}
\right]  \nonumber \\
&+&\frac{1}{\rho_{i}s}\int\limits_{x}^{\xi_i}\frac{d^{k}}{dx^{k}}\sin
\left[ \frac{s\left( x-y\right)}{\rho_{i}} \right]
q(y)\chi_{i\lambda}(y )dy \label{(4.22)}
\end{eqnarray}
for $x \in (\xi_n,b)$ and $(\xi_i,\xi_{i+1})$ (i=1,2,...r)
respectively.
\begin{thm} \label{(4.n)}
Let $Im \mu=t.$ Then \ if $\alpha_4\neq 0$
\begin{eqnarray}
\frac{d^{k}}{dx^{k}}\phi _{1\lambda}(x) &=&-\alpha_4s ^{2}\frac{d^{k}}{dx^{k}}%
\cos \left[ \frac{s\left(x-a\right)}{\rho_1} \right] +O\left( \left|
s\right| ^{k+1}e^{\left| t\right| \frac{(x-a)}{\rho_1}}\right)
\label{(lo2)} \\
\frac{d^{k}}{dx^{k}}\phi _{(i+1)\lambda}(x) &=&(-1)^{i+1}\alpha_4s^{i+2} \left(\prod_{j=1}^{i}\frac{\Delta_{j24}}{%
\Delta_{j12}\rho_j}\sin  \left[ \frac{s\left(
\xi_j-\xi_{j-1}\right)}{\rho_j} \right]\right)
\frac{d^{k}}{dx^{k}}\cos \left[ \frac{s\left(
x-\xi_i\right)}{\rho_{i+1}} \right] \nonumber
\\
&&+O\left(|s| ^{i+1} e^{\left|
t\right|((\sum\limits_{j=1}^{i}\frac{(\xi_j-\xi_{j-1})}{\rho_j})+\frac{(x-\xi_i)}{\rho_{i+1}})}\right),
\  \ i=1,2,...r \label{(4.p)}
\end{eqnarray}
as $\left| s \right| \rightarrow \infty $, while if $\alpha_4=0$%

\begin{eqnarray}
\frac{d^{k}}{dx^{k}}\phi _{1\lambda}(x) &=&-\alpha_3s\frac{d^{k}}{dx^{k}}%
\sin \left[ \frac{s\left(x-a\right)}{\rho_1} \right] +O\left( \left|
s\right| ^{k}e^{\left| t\right| \frac{(x-a)}{\rho_1}}\right)
\label{(lol)}\\
\frac{d^{k}}{dx^{k}}\phi _{2\lambda}(x) &=&-\alpha_3s^2\frac{\Delta_{124}}{%
\Delta_{112}}\frac{d^{k}}{dx^{k}}%
\cos \left[ \frac{s\left(\xi_1-a\right)}{\rho_1} \right]\cos \left[
\frac{s\left(x-\xi_1\right)}{\rho_2} \right]\nonumber\\ &+&O\left(
\left| s\right|^{k}e^{\left| t\right|
(\frac{(\xi_1-a)}{\rho_1}+\frac{(x-\xi_1)}{\rho_2})}\right)
\label{(lok)}
\\
\frac{d^{k}}{dx^{k}}\phi _{(i+1)\lambda}(x)
&=&(-1)^{i}\alpha_3s^{i+1} \cos \left[ \frac{s\left(
\xi_1-a\right)}{\rho_{1}} \right] \left(\prod_{j=2}^{i}\frac{\Delta_{(j-1)24}}{%
\Delta_{(j-1)12}\rho_{j}}\sin  \left[ \frac{s\left(
\xi_{j}-\xi_{j-1}\right)}{\rho_{j}} \right]\right)
\nonumber\\&\times&\frac{d^{k}}{dx^{k}}\cos \left[ \frac{s\left(
x-\xi_i\right)}{\rho_{i+1}} \right] +O\left(|s| ^{i} e^{\left|
t\right|((\sum\limits_{j=1}^{i}\frac{(\xi_j-\xi_{j-1})}{\rho_j})+\frac{(x-\xi_j)}{\rho_{j+1}})}\right)
\label{(c4ky)}
\end{eqnarray}
$ i=2,...r  \  as \left| s \right| \rightarrow \infty $ ($k=0,1)$.
Each of this asymptotic equalities hold uniformly for $x.$
\end{thm}

\begin{thm} \label{(c1)}
Let $Im s=t.$ Then \ if $\beta_4\neq 0$
\begin{eqnarray}
\frac{d^{k}}{dx^{k}}\chi _{(r+1)\lambda}(x) &=&\beta_4s ^{2}\frac{d^{k}}{dx^{k}}%
\cos \left[ \frac{s\left(b-x\right)}{\rho_{r+1}} \right] +O\left(
\left| s\right| ^{k+1}e^{\left| t\right|
\frac{(b-x)}{\rho_{r+1}}}\right)
\label{(c2)} \\
\frac{d^{k}}{dx^{k}}\chi _{(r-i)\lambda}(x)
&=&(-1)^{i+1}\beta_4s^{i+3}
\left(\prod_{j=0}^{i}\frac{\Delta_{(r-j)24}}{
\Delta_{(r-j)34}\rho_{r+1-j}}\sin  \left[ \frac{s\left(
\xi_{r+1-j}-\xi_{r-j}\right)}{\rho_{r+1-j}}
\right]\right)\nonumber\\&\times& \frac{d^{k}}{dx^{k}}\cos \left[
\frac{s\left( x-\xi_{r-i}\right)}{\rho_{r-i}} \right]+O\left(|s|
^{i+2} e^{\left|
t\right|((\sum\limits_{j=0}^{i}\frac{(\xi_{r+1-j}-\xi_{r-j})}{\rho_{r+1-j}})+\frac{(x-\xi_{r-j})}{\rho_{r-j}})}\right),
\label{(4.n1)}
\end{eqnarray}
 as $\left| s \right| \rightarrow \infty $, while if $\beta_4=0$%
\begin{eqnarray}
\frac{d^{k}}{dx^{k}}\chi _{(r+1)\lambda}(x) &=&-\beta_3\rho_{r+1}s \frac{d^{k}}{dx^{k}}%
\sin \left[ \frac{s\left(b-x\right)}{\rho_{r+1}} \right] +O\left(
\left| s\right| ^{k+1}e^{\left| t\right|
\frac{(b-x)}{\rho_{r+1}}}\right) \label{(c3)}
\\
\frac{d^{k}}{dx^{k}}\chi _{r\lambda}(x)
&=&-\beta_3s^2\frac{\Delta_{r24}}{\Delta_{r34} }\cos \left[
\frac{s\left( b-\xi_{r}\right)}{\rho_{r+1}} \right] \cos
\left[ \frac{s\left(x-\xi_r\right)}{\rho_{r}} \right]\nonumber\\
&+&O\left( \left| s\right| ^{k+1}e^{\left| t\right|(
\frac{(b-\xi_r)}{\rho_{r+1}}+\frac{(x-\xi_r)}{\rho_{r}})}\right)
\label{(hm)}
\\
\frac{d^{k}}{dx^{k}}\chi _{(r-i)\lambda}(x)
&=&(-1)^{i+1}\beta_3s^{i+2} \cos \left[ \frac{s\left(
b-\xi_{r}\right)}{\rho_{r+1}} \right] \frac{d^{k}}{dx^{k}}\cos
\left[ \frac{s\left(
x-\xi_{r-i}\right)}{\rho_{r-i}} \right]\nonumber\\&\times& \left(\prod_{j=0}^{i}\frac{\Delta_{(r-j)24}}{%
\Delta_{(r-j)34}\rho_{r-j}}\sin  \left[ \frac{s\left(
\xi_{r-1-j}-\xi_{r-j}\right)}{\rho_{r-j}} \right]\right)\nonumber\\
&+&O\left(|s| ^{i+1} e^{\left|
t\right|((\sum\limits_{j=0}^{i}\frac{(\xi_{r+1-j}-\xi_{r-j})}{\rho_{r+1-j}})+\frac{(x-\xi_{r-j})}{\rho_{r-j}})}\right)
\label{(c4)}
\end{eqnarray}
$as \left| s \right| \rightarrow \infty $  $ i=1,...r-1 $
($k=0,1)$.Each of this asymptotic equalities hold uniformly for $x.$
\end{thm}
\section{Asymptotic behaviour
 of eigenvalues and eigenfunctions}

Since the Wronskians of $\phi _{\lambda}(x )$ and $\chi _{\lambda}(x
)$ are independent of $x$ in each $\Omega_i (i=0,1,...,r+1)$, in
particular, by putting $x=b$ we have
\begin{eqnarray}\label{(ko)}
 \omega(\lambda)&=& \prod_{j=1}^{r}\frac{\Delta_{j12}}{
\Delta_{j34}}\omega_{r+1}(\lambda
)|_{x=b}=\prod_{j=1}^{r}\frac{\Delta_{j12}}{
\Delta_{j34}}\omega(\phi_{r+1}(b,\lambda ),\chi_{r+1}(b,\lambda
))\nonumber \\
&=& \prod_{j=1}^{r}\frac{\Delta_{j12}}{ \Delta_{j34}}
\{(\beta_1+\lambda\beta_3)\phi_{(n+1)}(b,\lambda
)+(\beta_2+\lambda\beta_4)\phi'_{(n+1)}(b,\lambda )\}.
\end{eqnarray}
 Let  $Im \mu=t.$ By substituting $(\ref{(c2)})$ and
$(\ref{(c4)})$ in $(\ref{(ko)})$ we obtain the following asymptotic
representations\\ \textbf{(i)} If $\alpha _{4}\neq 0$ and $\beta
_{4}\neq 0$, then
\begin{equation}\label{(4.15)}
w(\lambda )=-\alpha _{4}\beta _{4} s^{r+5}
\left(\prod\limits_{j=1}^{r+1}\frac{1}{\rho_{j}}\sin \left[
\frac{s\left( \xi_{j-1}-\xi_{j}\right)}{\rho_{j}} \right]\right)
+O\left(|s| ^{r+4} e^{\left|
t\right|(\sum\limits_{j=1}^{r+1}\frac{(\xi_{j}-\xi_{j-1})}{\rho_{j}})}\right)
\end{equation}
\textbf{(ii)} If $\alpha _{4}\neq  0$ and $\beta _{4}= 0$, then
\begin{eqnarray}
w(\lambda )&=&-\alpha _{4}\beta _{3} s^{r+4}\cos\left[ \frac{s\left(
b-\xi_r\right)}{\rho_{r+1}}
\right]\left(\prod\limits_{j=1}^{r}\frac{1}{\rho_{j}}\sin \left[
\frac{s\left( \xi_{j-1}-\xi_{j}\right)}{\rho_{j}} \right]\right)
\nonumber\\&+&O\left(|s| ^{r+3} e^{\left|
t\right|(\sum\limits_{j=1}^{r+1}\frac{(\xi_{j}-\xi_{j-1})}{\rho_{j}})}\right)
\label{(4.16)}
\end{eqnarray}
\textbf{(iii)} If $\alpha _{4}= 0$ and $\beta _{4}\neq 0$, then
\begin{eqnarray}
w(\lambda )&=&-\alpha _{3}\beta _{4} s^{r+4}\frac{\Delta_{112}}{
\Delta_{134}}\cos\left[ \frac{s\left( \xi_1-a\right)}{\rho_{1}}
\right]\left(\prod\limits_{j=2}^{r+1}\frac{1}{\rho_{j}}\sin \left[
\frac{s\left( \xi_{j-1}-\xi_{j}\right)}{\rho_{j}} \right]\right)
\nonumber\\&+&O\left(|s| ^{r+3} e^{\left|
t\right|(\sum\limits_{j=1}^{r+1}\frac{(\xi_{j}-\xi_{j-1})}{\rho_{j}})}\right)
\label{(4.54)}
\end{eqnarray}
\textbf{(iv)} If $\alpha _{4}=0$ and $\beta _{4}= 0$, then
\begin{eqnarray}
w(\lambda )&=&-\alpha _{3}\beta _{3} s^{r+3}\frac{\Delta_{112}}{
\Delta_{134}}\cos\left[ \frac{s\left( \xi_1-a\right)}{\rho_{1}}
\right] \cos\left[ \frac{s\left( b-\xi_r\right)}{\rho_{r+1}} \right]
\nonumber\\&\times&\left(\prod\limits_{j=2}^{r}\frac{1}{\rho_{j-1}}\sin
\left[ \frac{s\left( \xi_{j-1}-\xi_{j}\right)}{\rho_{j}}
\right]\right) +O\left(|s| ^{r+2} e^{\left|
t\right|(\sum\limits_{j=1}^{r+1}\frac{(\xi_{j}-\xi_{j-1})}{\rho_{j}})}\right)
\label{(4.54p)}
\end{eqnarray}
Now we are ready to derived the needed asymptotic formulas for
eigenvalues and  eigenfunctions.
\begin{thm}
The boundary-value-transmission problem $(\ref{1})$-$(\ref{5})$ has
an precisely numerable many real eigenvalues, whose behavior may be
expressed by $r+1$ sequence $\left\{ s _{n}^{j}\right\} (j=1,2,..r+1)$ with following asymptotic as $n\rightarrow \infty $ \\
\textbf{(i)} If $\alpha _{4} \neq 0$ and $\beta _{4}\neq 0 $, then
\begin{equation} \label{(5.1)}
s_{n}^{(j)}=(\frac{n}{2}-1)\frac{\rho_j\pi}{(\xi_{j-1}-\xi_j)}
+O\left( \frac{1}{n}\right), \ \ (j=1,2,...r+1)
\end{equation}
\textbf{(ii)} If $\alpha _{4} \neq 0$ and $\beta _{4}= 0$, then
\begin{equation}\label{(5.2)}
s_{n}^{(r+1)}=(n+\frac{1}{2})\frac{\rho_{r+1}\pi}{(b-\xi_r)}
+O\left( \frac{1}{n}\right), \
s_{n}^{(j)}=\frac{(n-1)\rho_j\pi}{2(\xi_{j-1}-\xi_j)}
+O\left( \frac{1}{n}\right), \ \ (j=1,...r),%
\end{equation}
\textbf{(iii)} If $\alpha _{4}=0$ and $\beta _{4}\neq 0$, then
\begin{equation}\label{(5.3)}
s_{n}^{(1)}=(n+\frac{1}{2})\frac{\rho_{1}\pi}{(\xi_1-a)}+O\left(
\frac{1}{n}\right), \
s_{n}^{(j)}=\frac{(n-1)\rho_j\pi}{2(\xi_{j-1}-\xi_j)}
+O\left(\frac{1}{n}\right), \ \ (j=2,...r+1),%
\end{equation}
\textbf{(iv)} If $\alpha _{4}=0$ and $\beta _{4}=0$, then
\begin{eqnarray} \label{(5.4)}
s_{n}^{(1)}&=&(n+\frac{1}{2})\frac{\rho_{1}\pi}{(\xi_1-a)} +O\left(
\frac{1}{n}\right), \
s_{n}^{(r+1)}=(n+\frac{1}{2})\frac{\rho_{r+1}\pi}{(b-\xi_r)}
+O\left( \frac{1}{n}\right),
\nonumber\\s_{n}^{(j)}&=&\frac{n\rho_j\pi}{2(\xi_{j-1}-\xi_j)}
+O\left( \frac{1}{n}\right), \ \ (j=2,...r),
\end{eqnarray}
\end{thm}
\begin{proof}
Let $\alpha _{4}\neq 0$ and $\beta _{4}\neq 0$. By applying the
well-known Rouche Theorem which asserts that if $f(z) \ \textrm{
and} \ g(z)$ are analytic inside and on a closed contour $\Gamma$,
and $|g(z)| < |f(z)|$ on $\Gamma$ then $f(z) \ \textrm{ and}  f(z) +
g(z)$ have the same number zeros inside $\Gamma$ provided that the
zeros are counted with multiplicity on a sufficiently large contour,
it follows that $w(\lambda)$ has the same number of zeros inside the
suitable contour as the leading term $w_{0}(\lambda)=-\alpha
_{4}\beta _{4} s^{r+5}
\left(\prod\limits_{j=1}^{r+1}\frac{1}{\rho_{j}}\sin \left[
\frac{s\left( \xi_{j-1}-\xi_{j}\right)}{\rho_{j}} \right]\right)$ in
$(\ref{(4.15)})$. Hence, if $\lambda_{0} < \lambda_{1}< \lambda_{2}
. . .$ are the zeros of $w(\lambda)$ and $\lambda_{n}$, we have the
needed asymptotic formulas $(\ref{(5.1)})$. Other cases can be
proved similarly.
\end{proof}
Using this asymptotic expressions of eigenvalues we can obtain the
corresponding asymptotic expressions for eigenfunctions of the
problem $(\ref{1})$-$(\ref{4})$.
 Recalling that $\phi_{\lambda _{n}}(x)$ is an eigenfunction
 according to the eigenvalue $\lambda_{n},$ and by putting (\ref{(5.1)}) in the (\ref{(lo2)})-(\ref{(4.p)}) for $k=0,1$
 and denoting the corresponding  eigenfunction as
 $\phi_{n}^{(k)}(x) \ k=1,2,...r+1$
we get the following cases If $\alpha _{4} \neq 0$ and $\beta
_{4}\neq 0 $, then
\begin{eqnarray*}
\phi_{n}^{(k)}(x)=\left\{
\begin{array}{ll}
\begin{array}{l}
-\alpha_4\left[(\frac{n}{2}-1)\frac{\rho_k\pi}{(\xi_{k-1}-\xi_k)}\right]^{2}%
\cos \left[
(\frac{n}{2}-1)\frac{\rho_k\pi}{(\xi_{k-1}-\xi_k)}\frac{\left(x-a\right)}{\rho_1}
\right] +O(n),
\end{array}
\begin{array}{l}
x\in ( a,\xi_1)
\end{array}
\\
(-1)^{i+1}\alpha_4\left[(\frac{n}{2}-1)\frac{\rho_k\pi}{(\xi_{k-1}-\xi_k)}\right]^{i+2}\left(\prod_{j=1}^{i}\frac{\Delta_{j24}}{%
\Delta_{j12}\rho_j}\sin  \left[ (\frac{n}{2}-1)\frac{\rho_k\pi (
\xi_j-\xi_{j-1})}{(\xi_{k-1}-\xi_k)\rho_j} \right]\right) \nonumber
\\ \times \cos \left[
(\frac{n}{2}-1)\frac{\rho_k\pi(
x-\xi_i)}{(\xi_{k-1}-\xi_k)\rho_{i+1}} \right] +O(n^{i+1}), \ \ x\in
(\xi_i,\xi_{i+1}), \ i=1,2,...r
\end{array}
\right.
\end{eqnarray*}
where k=1,2,...r+1.
  If $\alpha _{4} \neq 0$ and $\beta _{4}= 0 $
\begin{eqnarray*}
\phi_{n}^{(r+1)}(x)=\left\{
\begin{array}{ll}
\begin{array}{l}
-\alpha_4\left[(n+\frac{1}{2})\frac{\rho_{r+1}\pi}{(b-\xi_n)}\right]^{2}%
\cos \left[
(n+\frac{1}{2})\frac{\rho_{r+1}\pi}{(b-\xi_n)}\frac{\left(x-a\right)}{\rho_1}
\right] +O(n),
\end{array}
\begin{array}{l}
x\in ( a,\xi_1)
\end{array}
\\
(-1)^{i+1}\alpha_4\left[(n+\frac{1}{2})\frac{\rho_{r+1}\pi}{(b-\xi_n)}\right]^{i+2}\left(\prod_{j=1}^{i}\frac{\Delta_{j24}}{%
\Delta_{j12}\rho_j}\sin \left[ (\frac{1}{2}+n)\frac{\rho_{r+1}\pi(
\xi_j-\xi_{j-1})}{(b-\xi_n)\rho_j} \right]\right) \nonumber
\\ \times \cos \left[
(\frac{1}{2}+n)\frac{\rho_{r+1}\pi( x-\xi_i)}{(b-\xi_n)\rho_{i+1}}
\right] +O(n^{i+1}), \ x\in (\xi_i,\xi_{i+1}) \ \ i=1,2,...r
\end{array}
\right.
\end{eqnarray*}
and
\begin{eqnarray*}
\phi_{n}^{(k)}(x)=\left\{
\begin{array}{ll}
\begin{array}{l}
-\alpha_4\left[\frac{(n-1)\rho_k\pi}{2(\xi_{k-1}-\xi_k)}\right]^{2}%
\cos \left[
\frac{(n-1)\rho_k\pi}{2(\xi_{k-1}-\xi_k)}\frac{\left(x-a\right)}{\rho_1}
\right] +O(n),
\end{array}
\begin{array}{l}
x\in ( a,\xi_1)
\end{array}
\\
(-1)^{i+1}\alpha_4\left[\frac{(n-1)\rho_k\pi}{2(\xi_{k-1}-\xi_k)}\right]^{i+2}\left(\prod_{j=1}^{i}\frac{\Delta_{j24}}{%
\Delta_{j12}\rho_j}\sin  \left[ \frac{(n-1)\rho_k\pi (
\xi_j-\xi_{j-1})}{2(\xi_{k-1}-\xi_k)\rho_j} \right]\right) \nonumber
\\ \times \cos \left[
\frac{(n-1)\rho_k\pi( x-\xi_i)}{2(\xi_{k-1}-\xi_k)\rho_{i+1}}
\right] +O(n^{i+1}), \  \ i=1,2,...r \textrm{ for }x\in
(\xi_i,\xi_{i+1})
\end{array}
\right.
\end{eqnarray*}
where k=1,2,...r. If $\alpha _{4} = 0$ and $\beta _{4}\neq 0 $, then
\begin{eqnarray*}
\phi_{n}^{(1)}(x)=\left\{
\begin{array}{ll}
\begin{array}{l}
-\alpha_3\left[
(n+\frac{1}{2})\frac{\rho_{1}\pi}{(\xi_1-a)}\right]\sin \left[
(n+\frac{1}{2})\frac{\rho_{1}\pi(x-a)}{(\xi_1-a)\rho_1}\right]
+O(1),
\end{array}
\begin{array}{l}
x\in ( a,\xi_1)
\end{array}
\\
-\alpha_3\left[
(n+\frac{1}{2})\frac{\rho_{1}\pi}{(\xi_1-a)}\right]^2\frac{\Delta_{124}}{
\Delta_{112}} \cos \left[ (n+\frac{1}{2})\pi\right]\cos \left[
\frac{(n+\frac{1}{2})\rho_{1}\pi(x-\xi_1)}{(\xi_1-a)\rho_2} \right]
+O(n)\\ \textrm{ for }x\in
(\xi_1,\xi_{2})\\
\begin{array}{l}
(-1)^{i}\alpha_3\left[
(n+\frac{1}{2})\frac{\rho_{1}\pi}{(\xi_1-a)}\right]^{i+1} \cos
\left[(n+\frac{1}{2})\pi\right]\cos
\left[(n+\frac{1}{2})\frac{\rho_{1}\pi(
x-\xi_i)}{(\xi_1-a)\rho_{i+1}}\right]\nonumber\\
\times
\left(\prod_{j=2}^{i}\frac{\Delta_{(j-1)24}}{%
\Delta_{(j-1)12}\rho_{j}}\sin \left[
(n+\frac{1}{2})\frac{\rho_{1}\pi(
\xi_{j}-\xi_{j-1})}{(\xi_1-a)\rho_{j}} \right]\right)
 +O(n^{i}) \  x\in
(\xi_i,\xi_{i+1})\\i=2,...r \
\end{array}
\end{array}
\right.
\end{eqnarray*}
and
\begin{eqnarray*}
\phi_{n}^{(k)}(x)=\left\{
\begin{array}{ll}
\begin{array}{l}
-\alpha_3\left[ \frac{(n-1)\rho_k\pi}{2(\xi_{k-1}-\xi_k)}\right]\sin
\left[\frac{(n-1)\rho_k\pi(x-a)}{2(\xi_{k-1}-\xi_k)\rho_1} \right]
+O(1),
\end{array}
\begin{array}{l}
x\in ( a,\xi_1)
\end{array}
\\
-\alpha_3\left[
\frac{(n-1)\rho_k\pi}{2(\xi_{k-1}-\xi_k)}\right]^2\frac{\Delta_{124}}{%
\Delta_{112}} \cos \left[
\frac{(n-1)(\xi_1-a)\rho_k\pi}{2(\xi_{k-1}-\xi_k)\rho_1} \right]\cos
\left[ \frac{(n-1)(x-\xi_1)\rho_k\pi}{2(\xi_{k-1}-\xi_k)\rho_2}
\right] +O(n)\\ \textrm{ for }x\in
(\xi_1,\xi_{2})\\
\begin{array}{l}
(-1)^{i}\alpha_3\left[
\frac{(n-1)\rho_k\pi}{2(\xi_{k-1}-\xi_k)}\right]^{i+1} \cos \left[
\frac{(n-1)(\xi_1-a)\rho_k\pi}{2(\xi_{k-1}-\xi_k)\rho_{1}}
\right]\cos \left[\frac{(n-1)(
x-\xi_i)\rho_k\pi}{2(\xi_{k-1}-\xi_k)\rho_{i+1}} \right]\nonumber\\ \times \left(\prod_{j=2}^{i}\frac{\Delta_{(j-1)24}}{%
\Delta_{(j-1)12}\rho_{j}}\sin  \left[\frac{(n-1)(
\xi_{j}-\xi_{j-1})\rho_k\pi}{2(\xi_{k-1}-\xi_k)\rho_{j}}
\right]\right)
 +O(n^{i}), \  \  \   x\in
(\xi_i,\xi_{i+1})\\i=2,...r  \  \
\end{array}
\end{array}
\right.
\end{eqnarray*}
where k=2,...r+1. If $\alpha _{4} = 0$ and $\beta _{4}= 0 $, then
\begin{eqnarray*}
\phi_{n}^{(1)}(x)=\left\{
\begin{array}{ll}
\begin{array}{l}
-\alpha_3\left[
(n+\frac{1}{2})\frac{\rho_{1}\pi}{(\xi_1-a)}\right]\sin \left[
(n+\frac{1}{2})\frac{\rho_{1}\pi(x-a)}{(\xi_1-a)\rho_1}\right]
+O(1),
\end{array}
\begin{array}{l}
x\in ( a,\xi_1)
\end{array}
\\
-\alpha_3\left[
(n+\frac{1}{2})\frac{\rho_{1}\pi}{(\xi_1-a)}\right]^2\frac{\Delta_{124}}{
\Delta_{112}} \cos \left[ (n+\frac{1}{2})\pi\right]\cos \left[
\frac{(n+\frac{1}{2})\rho_{1}\pi(x-\xi_1)}{(\xi_1-a)\rho_2} \right]
+O(n)\\ \textrm{ for }x\in (\xi_1,\xi_{2})\\
\begin{array}{l}
(-1)^{i}\alpha_3\left[
(n+\frac{1}{2})\frac{\rho_{1}\pi}{(\xi_1-a)}\right]^{i+1} \cos
\left[(n+\frac{1}{2})\pi\right]\cos
\left[(n+\frac{1}{2})\frac{\rho_{1}\pi(
x-\xi_i)}{(\xi_1-a)\rho_{i+1}}\right]\nonumber\\
\times
\left(\prod_{j=2}^{i}\frac{\Delta_{(j-1)24}}{%
\Delta_{(j-1)12}\rho_{j}}\sin \left[
(n+\frac{1}{2})\frac{\rho_{1}\pi(
\xi_{j}-\xi_{j-1})}{(\xi_1-a)\rho_{j}} \right]\right)
 +O(n^{i}), \    x\in
(\xi_i,\xi_{i+1})\\i=2,...r \  \
\end{array}
\end{array}
\right.
\end{eqnarray*}
\begin{eqnarray*}
\phi_{n}^{(r+1)}(x)=\left\{
\begin{array}{ll}
\begin{array}{l}
-\alpha_3\left[
(n+\frac{1}{2})\frac{\rho_{r+1}\pi}{(b-\xi_n)}\right]\sin \left[
(n+\frac{1}{2})\frac{\rho_{r+1}\pi(x-a)}{(b-\xi_n)\rho_1}\right]
+O(1),
\end{array}
\begin{array}{l}
x\in ( a,\xi_1)
\end{array}
\\
-\alpha_3\left[
(n+\frac{1}{2})\frac{\rho_{r+1}\pi}{(b-\xi_n)}\right]^2\frac{\Delta_{124}}{
\Delta_{112}} \cos \left[
\frac{(n+\frac{1}{2})\rho_{r+1}\pi(\xi_1-a)}{(b-\xi_n)\rho_1}
\right]\cos \left[
\frac{(n+\frac{1}{2})\rho_{r+1}\pi(x-\xi_1)}{(b-\xi_n)\rho_2}
\right]\\ +O(n) \textrm{ for }x\in (\xi_1,\xi_{2})\\
\begin{array}{l}
(-1)^{i}\alpha_3\left[
(n+\frac{1}{2})\frac{\rho_{r+1}\pi}{(\xi_1-a)}\right]^{i+1}  \cos
\left[ \frac{(n+\frac{1}{2})\rho_{r+1}\pi(\xi_1-a)}{(b-\xi_n)\rho_1}
\right]\cos \left[(n+\frac{1}{2})\frac{\rho_{r+1}\pi(
x-\xi_i)}{(b-\xi_n)\rho_{i+1}}\right]\nonumber\\
\times
\left(\prod_{j=2}^{i}\frac{\Delta_{(j-1)24}}{%
\Delta_{(j-1)12}\rho_{j}}\sin \left[
\frac{(n+\frac{1}{2})\rho_{r+1}\pi(
\xi_{j}-\xi_{j-1})}{(b-\xi_n)\rho_{j}} \right]\right)
 +O(n^{i}), \  \  \    x\in
(\xi_i,\xi_{i+1})\\i=2,...r \  \
\end{array}
\end{array}
\right.
\end{eqnarray*}
and
\begin{eqnarray*}
\phi_{n}^{(k)}(x)=\left\{
\begin{array}{ll}
\begin{array}{l}
-\alpha_3\left[ \frac{(n-1)\rho_k\pi}{2(\xi_{k-1}-\xi_k)}\right]\sin
\left[\frac{(n-1)\rho_k\pi(x-a)}{2(\xi_{k-1}-\xi_k)\rho_1} \right]
+O(1),
\end{array}
\begin{array}{l}
x\in ( a,\xi_1)
\end{array}
\\
-\alpha_3\left[
\frac{n\rho_k\pi}{2(\xi_{k-1}-\xi_k)}\right]^2\frac{\Delta_{124}}{%
\Delta_{112}} \cos \left[
\frac{n(\xi_1-a)\rho_k\pi}{2(\xi_{k-1}-\xi_k)\rho_1} \right]\cos
\left[ \frac{n(x-\xi_1)\rho_k\pi}{2(\xi_{k-1}-\xi_k)\rho_2} \right]
+O(n)\\
\begin{array}{l}
(-1)^{i}\alpha_3\left[ \frac{n
\rho_k\pi}{2(\xi_{k-1}-\xi_k)}\right]^{i+1} \cos \left[
\frac{n(\xi_1-a)\rho_k\pi}{2(\xi_{k-1}-\xi_k)\rho_{1}} \right]\cos
\left[\frac{n(
x-\xi_i)\rho_k\pi}{2(\xi_{k-1}-\xi_k)\rho_{i+1}} \right]\nonumber\\ \times \left(\prod_{j=2}^{i}\frac{\Delta_{(j-1)24}}{%
\Delta_{(j-1)12}\rho_{j}}\sin  \left[\frac{n(
\xi_{j}-\xi_{j-1})\rho_k\pi}{2(\xi_{k-1}-\xi_k)\rho_{j}}
\right]\right)
 +O(n^{i}), \  \  \    x\in
(\xi_i,\xi_{i+1})\\i=2,...r \  \
\end{array}
\end{array}
\right.
\end{eqnarray*} where  k=2,...r+1.
All this asymptotic approximations are hold uniformly for $x.$


\begin{thebibliography}{99}

\bibitem{lik}A. V. Likov and Yu. A. Mikhailov, {\em  The theory of Heat and Mass
Transfer}, Qosenergaizdat, 1963(Russian).


\bibitem{osh2}O. Sh. Mukhtarov and H. Demir,
{\em Coersiveness of the discontinuous initial- boundary value
problem for parabolic equations}, Israel J. Math., Vol. 114(1999),
Pages  239{-}252.


\bibitem{osh6}M. Kadakal and O. Sh. Mukhtarov
{\em Sturm–Liouville problems with discontinuities at two points},
Computers and Mathematics with Applications 54(2007) 1367{-}1379

\bibitem{ka} F. S. Muhtarov and K. Aydemir{\em Distributions of eigenvalues for Sturm-Liouville
problem under jump conditions}, Journal of New Results in Science
1(2012) 81-89.

\bibitem{osh3}O. Sh. Mukhtarov and  S. Yakubov,
{\em Problems for ordinary differential equations with transmission
conditions}, Appl. Anal.,81(2002),1033{-}1064.


\bibitem{rasu1}M. L. Rasulov, {\em Methods of Contour Integration}, North-Holland Publishing Company, Amsterdam, 1967.


\bibitem{tik}A. N. Tikhonov and A. A. Samarskii, {\em Equations of Mathematical Physics, }
Oxford and New York, Pergamon, 1963.

\bibitem{titc}E. C. Titchmarsh, {\em
Eigenfunctions Expansion Associated with Second Order Differential
Equations I}, second edn. Oxford Univ. Press, London, 1962.

\bibitem{tite}I. Titeux and Ya. Yakubov,
{\em Completeness of root functions for thermal conduction in a
strip with piecewise continuous coefficients}, Math. Models Methods
Appl. Sc., 7(7), (1997), 1035-1050.


\bibitem{voito}N. N. Voitovich , B. Z. Katsenelbaum and A. N. Sivov ,
{\em Generalized Method of Eigen-vibration in the theory of
Diffraction }, Nakua, Mockow, 1997 (Russian).

\end{thebibliography}
\end{document}